\documentclass{article}
\usepackage{amsfonts}
\usepackage[mathscr]{eucal}
 \usepackage{amsmath}
 \usepackage{amssymb}

\usepackage[T2A]{fontenc}
\usepackage[cp1251]{inputenc}
\usepackage[english]{babel}

\newtheorem{theorem}{Theorem}[section]

\newtheorem{proposition}[theorem]{Proposition}
\newtheorem{definition}[theorem]{Definition}
\newtheorem{remark}[theorem]{Remark}

\newtheorem{assumption}[theorem]{Assumption}

\newcommand{\Rset}{\mathbb R}
\newcommand{\R}{\mathbb R}
\newcommand{\Nset}{\mathbb N}

\newcommand{\al}{{\alpha}}

\newcommand{\Om}{{\Omega}}

\newcommand{\va}{{\varphi}}
\newcommand{\vr}{{\varrho}}
\newcommand{\la}{\lambda}

\newcommand{\ga}{{\gamma}}

\newcommand{\cCL}{{\cal C\hskip-0.5mm L}}

\newcommand{\fA}{\mathfrak{A}}

\newcommand{\hf}{{\frac12}}

\newcommand{\wrt}{{with respect to }}

\newenvironment{declaration}[1]{\trivlist
\item[\hskip \labelsep{\bf #1 }]\ignorespaces}{\endtrivlist}
\newenvironment{proofof}[1]{\begin{declaration}{#1}}{\hfill
$\square$ \end{declaration}}
\newenvironment{proof}{\begin{proofof}{Proof.}}{\end{proofof}}

\begin{document}

\title{Finite-dimensional global attractors  for parabolic
nonlinear equations with state-dependent delay}
\author{Igor Chueshov$^{a,}$\footnote{\footnotesize 
E-mails:  chueshov@karazin.ua (I.Chueshov),
rezounenko@yahoo.com (A.Rezounenko)}
 ~~ and ~~ Alexander  Rezounenko$^{a,b}$ \\ \\
 $^a$Department of Mechanics and Mathematics, \\
 Karazin Kharkov National University,  Kharkov, 61022,  Ukraine\\
$^b$Institute of Information Theory and Automation, Academy of
Sciences \\ of the Czech Republic, P.O. Box 18, 182\,08 Praha, CR }
  \maketitle
\begin{abstract}
We deal with a class of parabolic  nonlinear evolution
equations with state-dependent delay.
This class covers several important PDE models arising in biology.
We first prove well-posedness in a
certain space of functions which are Lipschitz  in time. We  show that the model
considered
 generates an evolution operator semigroup $S_t$ on a  space $\cCL$ of Lipschitz type
 functions over delay time interval.
The operators $S_t$ are closed  for all $t\ge 0$ and continuous for $t$ large enough.
Our main result shows that the semigroup $S_t$  possesses
 compact global and exponential attractors of finite fractal dimension.
 Our argument is based on the recently developed method of quasi-stability
 estimates
and involves  some extension of the theory of global attractors for the case of closed
evolutions.
\end{abstract}

\section{Introduction}


Differential equations with different types of delay attract much
attention during last decades. Including delay terms
in differential equations is a natural step  of taking into account that
the majority of real-world problems depends on the pre-history of the
evolution. Delay terms in an equation reflect a well-understood
phenomenon that evolution  of a state of a system depends not
only on this state but rather on the states during some
previous interval of time (memory of the system).
This leads to infinite-dimensional dynamics even in the case
 of ordinary differential equations.
 The general theory
of delay differential equations was initially developed for the
simplest case of constant delays. We cite just classical monographs
\cite{Bellman-Cooke_AP-1963_book,Walther_book,Hale} on ordinary
differential equations (ODEs) and milestone articles
\cite{Fitzgibbon-JDE-1978,Travis-Webb_TAMS-1974} on partial
differential equations (PDEs) with constant delays. On the other
hand it is clear that the constancy of the delay is just an extra
assumption made to simplify the study, but it is not really well-motivated by
real-world models. To describe a process more naturally a  new class
of {\it state-dependent delay} models was introduced and intensively
studied during last decades. We mention works on ODEs
\cite{Driver-AP-1963,Hartung-Krisztin-Walther-Wu-2006,Krisztin-Arino_JDDE-2001,Walther_JDE-2003}
and on PDEs
\cite{ChuRez2013-sdd,Rezounenko_JMAA-2007,Rezounenko_NA-2009,Rezounenko_JMAA-2012,Rezounenko-Zagalak-DCDS-2013} with state-dependent delays.
\par
 The simplest case of a state-dependent delay is a delay explicitly given by a real-valued function $\eta : \Rset \to \Rset_{+}$ which depends  on the value $x(t)$
  at the reference time $t$
  but not on previous  values of the
solution $\{ x(\tau), \tau\le t\}$. This leads to terms of the form
$f(x(t-\eta(x(t)))$ in the model considered. Even in this  case the
non-uniqueness could appear (see the scalar ODE example constructed
by R.Driver \cite{Driver-AP-1963} in 1963 for initial data from the
space of continuous functions on the delay interval). The standard
way for general models to avoid non-uniqueness in the case of
infinite-dimensional dynamics is to consider smoother (narrower) classes of
solutions. However in this case the existence problem may become
critical. The main task is to find a good balance between these two
issues.
\par
In this paper we  deal with      a certain abstract parabolic problem with the
state dependent delay term of a rather general structure. Our considerations are motivated by
several biological models,
see the discussion and the references in \cite{Britton-1990},
\cite{Gourley-So-Wu-2003}
and \cite{Rezounenko-Zagalak-DCDS-2013}.
 Our main goal in this paper is to find appropriate phase
spaces in which we can establish the well-posedness of our model and
study its long time (qualitative) dynamics.
\par
Our first result (Theorem \ref{th:wp})  states  well-posedness of
the problem and allows us to define an evolution semigroup $S_t$ of
closed mappings on a certain Banach space of functions on the delay
time  interval with values in an appropriate Hilbert space. In some
sense this result extends the well-posedness statements in
\cite{Rezounenko_NA-2009,Rezounenko_JMAA-2012,Rezounenko-Zagalak-DCDS-2013}
to more general delay terms.
 The main result of the paper  (Theorem \ref{th:attr}) states the existence of a global
 \emph{ finite-dimensional} attractor, the object which is responsible for long-time dynamics.
 We also show that the model possesses an exponential fractal global attractor
 (see the definition in the Appendix).
 \par
 Although
for some parabolic problems
 with state-dependent delay the existence of compact global attractors was established earlier in
\cite{Rezounenko_NA-2009,Rezounenko-Zagalak-DCDS-2013}, to the best
of our knowledge,  results on \textit{finite-dimensional} behavior
for parabolic  state-dependent delay problems were not known before.
The main difficulty is related to the fact the corresponding delay
term is not Lipschitz on the natural energy balance space. We also
mention that our Theorem \ref{th:attr} can be applied in the
situation considered in \cite{Rezounenko-Zagalak-DCDS-2013}  and
gives the finite-dimensionality of the global attractor constructed
in that paper.

\par
We note that the evolution operators $S_t$  we construct are not
continuous mapping on the phase space for $t$ small enough.
Therefore to prove the existence of a compact global attractor  we
use the extension of the standard  theory suggested in
\cite{PataZelik-cls}. As for dimension issues we
  apply the idea of the method
of quasi-stability estimates developed earlier in \cite{chlJDE04,Chueshov-Lasiecka-MemAMS-2008_book,Chueshov-Lasiecka-2010_book,CL-hcdte-notes}
 for the second order in time evolution models which generate continuous evolution semigroups. This is possible in our case due to the continuity of evolution operator for large times. We note that
in the delay case the quasi-stability method was applied earlier in \cite{CLW-delay,CLW-delay-surv,ChuRez2013-sdd}
for second  order models, see also  \cite[Chapter 6]{Chueshov-2014_book}.

\section{Model description}
We deal with
 well-posedness and  long-time  dynamics of
abstract evolution equations with delay  of the form
\begin{equation}\label{main-eq}
\dot u(t) +A u(t) + F(u_t)+G(u(t)) = h, \quad t>0,
\end{equation}
in some Hilbert space $H$. Here the dot over an element means time
derivative, $A$ is  linear and $F$, $G$ are nonlinear operators, $h\in H$.
The term
$F(u_t)$ represents (nonlinear) delay effect in the dynamics.
As usually for delay
equations, the history segment (the state) is denoted by  $u_t\equiv
u_t(\theta)\equiv u(t+\theta)$ for $\theta \in [-r,0].$
\begin{assumption}[Basic Hypotheses]\label{as:basic}
{\rm In our study we assume that:
\begin{itemize}
  \item [{\bf (A)}]    $A$  is a positive operator
with a discrete spectrum in a separable Hilbert space $H$ with
a dense domain $D(A)\subset H$.  Hence there exists an orthonormal basis $\{ e_k \}$
of $ H$ such that $$ A e_k = \la_k e_k,\quad\mbox{with } 0< \la_1\le
 \la_2\le \dots ,\quad \ \lim_{k \to \infty}  \la_k=\infty.
 $$
We  define the spaces $H_\al$  which are $D(A^{\alpha})$ for $\al\ge0$
 (the domain of $A^{\alpha}$) and the completions of $H$ \wrt the norm $\|A^\al\cdot\|$ when $\al<0$
 (see, e.g., \cite{Lions-Magenes-book}).
 Here and below, $\parallel \cdot\parallel$ is the norm of $H$, and
$\langle \cdot, \cdot \rangle $ is the corresponding scalar  product.
   For
$r>0,$ we denote for short $C_{\alpha}=C([-r,0]; H_\al)$ which
is a Banach space with the  norm
$$
\vert v
\vert_{C_{\alpha}} \equiv \sup \{\parallel   v(\theta
)\parallel_\al : \theta \in [-r,0] \},
$$
where $\|v\|_\al=\|A^{\al}v\|$ is the norm in $H_\al$ for $\al\in\R$.
We  also write $C=C_0$.
  \item [{\bf (F)}] The delay term $F(u_t)$ has
  the form $F(u_t)\equiv F_0(u(t-\eta(u_t)))$,
  where
\textbf{(a)} $F_0 : H_\al\mapsto H_\al$
  is globally Lipschitz for $\al=0$ and   $\al =-1/2$, i.e.,
there exists $L_F>0$ such that
\begin{equation}\label{F-lip}
    ||F_0(v)- F_0(u)||_\al \le L_{F}
|| v-u||_\al,~~~ v,u\in H_\al,~~\al=0,-1/2;
\end{equation}
and \textbf{(b)} $\eta :C\equiv C([-r,0]; H)\mapsto [0,r]\subset \R$ is globally Lipschitz:
\begin{equation}\label{eta-lip}
    |\eta(\phi)- \eta (\psi)| \le L_{\eta}
|\phi-\psi|_{C},~~~ \phi,\psi\in C([-r,0]; H).
\end{equation}
 \item [{\bf (G)}]
$G : H_{1/2}\mapsto H$  is
locally Lipschitz, i.e.
\begin{equation}\label{G-lip}
    ||G(v)- G(u)|| \le L_G(R) || v-u||_{1/2},
    ~~~ v,u\in H_{1/2},~~ \|v\|_{1/2}, \|u\|_{1/2}\le R,
\end{equation}
where $L_G: \R_{+} \to  \R_{+}$ is a nondecreasing function.
 We also assume that $G$ is a
  potential mapping, the latter means that
 there exists a
(Frech\'{e}t differentiable) functional $\Pi(u) : H_{1/2}\to \R$
such that $G(u)=\Pi^\prime (u)$ in the sense
$$ \lim_{||v||_{1/2} \to 0} ||v||^{-1}_{1/2} \big[
\Pi (u+v) - \Pi(u) + \langle (u),v\rangle \big] =0.
$$
Moreover, we assume that \textbf{(a)}  there exist positive constants
 $c_1$ and $c_2$ such that
\begin{equation}\label{sdd17-G}
\langle G(u),Au\rangle\ge  -c_1||A^{\frac{1}{2}}u||^2-c_2, \quad u\in
D(A);
\end{equation}
and \textbf{(b)} there exist $\delta>0$ and $m\ge 0$ such that $G :
H_{1/2-\delta}\mapsto H_{-m}$ is continuous.
\end{itemize}
}
\end{assumption}

Our  main motivating  example of a system with  discrete state-dependent delay
is the following one:
\begin{equation}\label{sdd9-1}
\frac{\partial }{\partial t}
u(t,x)-\Delta u(t,x) +  b\left(
[Bu(t-\eta (u_t),\cdot)](x)\right)+g(u(t,x))=h(x),\quad x\in \Omega,~~t>0,
\end{equation}
in a bounded domain $\Om \subset \R^n$, where $B: L^2(\Omega) \to
L^2(\Omega)$ is a  bounded operator and $b:\mathbb{R}\to \mathbb{R}$
stands for a Lipschitz map. The function $\eta : C([-r,0];
L^2(\Omega)) \to [0,r]\subset \mathbb{R}_{+}$ denotes \textit{a
state-dependent discrete delay}. The Nemytskii operator $u\mapsto
g(u)$ with $C^1$ function $g$ represents a nonlinear non-delayed reaction
term and $h(x)$ describes sources.
 The form of the delay term  is motivated by models in population dynamics
 where function $b$ is a birth function (could be
 $b(s)=c_1 s \cdot e^{-c_2 s}$, with $c_1,c_2>0$) and the delay $\eta$
 represents the maturity age. For more detailed discussion and further
 examples (the diffusive Nicholson blowflies equation,
 Mackey-Glass equation - the diffusive model of Hematopoiesis  - blood cell
 production, the Lasota-Wazewska-Czyzewska model in
hematology) with state-dependent delay
 we refer to  \cite{Gourley-So-Wu-2003}
and \cite{Rezounenko-Zagalak-DCDS-2013}  and to the references
therein. We
note that several special cases of
the model in \eqref{sdd9-1} were studied in
\cite{Rezounenko_NA-2009,Rezounenko_NA-2010,Rezounenko_JMAA-2012,Rezounenko-Zagalak-DCDS-2013}).
For instance it was assumed in
 \cite{Rezounenko-Zagalak-DCDS-2013}
that $g(s)\equiv 0$, $b(s)$ is a bounded function, and  $B$ is an integral compact linear operator. This
leads to nonlocal  (in space) models. Our assumptions covers the
non-compact case. We can take $b(s)=s$ and $B=Id$, for instance.
We also note that if
we  equip \eqref{sdd9-1} with the Dirichlet boundary condition, then
the dissipativity property in \eqref{sdd17-G} holds
  provided $g\in C^1(\R)$, $g(0)=0$ and the derivative $g'(s)$ is bounded from below. This
follows by the standard integration by parts. Thus population dynamics models
with nonlinear sink/source feedback terms can be included in consideration.  For this kind of biological models, but
with state-{\it independent} delay, we refer to \cite{Wu-book}.
\smallskip\par

We equip the equation \eqref{main-eq} with the initial condition
\begin{equation}\label{sdd9-ic}
   u(\theta)=\varphi(\theta),~~\theta\in[-r,0],
   \end{equation}
and for initial data $\va$ consider the space
%
%
\begin{equation}\label{sdd9-8}
\cCL \equiv \left\{ \varphi \in C([-r,0]; H) \, \left|\; {\rm
Lip}_{[-r,0]}(A^{-\hf}\va) <\!+\infty;\,\right. \varphi(0)\in
D(A^{\frac{1}{2}})\right\},
\end{equation}
where
\begin{equation*}
{\rm Lip}_{[a,b]}(\varphi)\equiv  \sup\limits_{s\neq t} \left\{
\frac{||\varphi (s)-\varphi (t) ||}{|s-t|} \,:\; s,t\in [a,b],~ s\neq t
\right\}
\end{equation*}
denotes the corresponding Lipschitz constant. One can show that the
space $\cCL$ consists of continuous functions $\varphi$ on $[-r,0]$
with values in $H$  such that $\varphi(0)\in H_{1/2}$ and which are
absolutely continuous in $H_{-1/2}$. The latter means that there
exists the derivative $\dot\varphi\in L^\infty(-r,0; H_{-1/2})$ such
that
\[
\va(s)=\va(0)-\int_s^0\dot\va(\xi) d\xi,  ~~s\in [-r,0],
\]
and
\[
{\rm Lip}_{[-r,0]}(A^{-\hf}\varphi) = {\rm ess\, sup}\left\{
 ||A^{-\frac{1}{2}}\dot \varphi (s)||\, :\; s\in [-r,0]\right\}\equiv
 |\dot\phi|_{L^\infty(-r,0; H_{-1/2})}.
\]
We equip the space $\cCL$ with the natural norm
\begin{equation}\label{sdd9-9}
|\varphi|_{\cCL}\equiv \max_{s\in [-r,0]} ||\varphi(s)|| +{\rm
Lip}_{[-r,0]}(A^{-\hf}\varphi) + ||A^{\frac{1}{2}} \varphi(0)||.
\end{equation}
We note that the delay term $F(\va)\equiv F_0(\va(-\eta(\va)))$ in
\eqref{main-eq} is well-defined for every $\varphi\in C$ and
possesses the property
(see \eqref{F-lip} for $\alpha =0$) $||F(\varphi) -F(0) || \le
L_F ||\varphi (-\eta(\varphi)) || \le L_F |\varphi|_C$, hence
\begin{equation}\label{lin-b-F}
||F(\varphi)|| \le c_1+ c_2 |\varphi|_C,~~~\va\in C,
\end{equation}
with $c_1=||F(0)||$ and $c_2=L_F$.
 However it is not Lipschitz on the space $C$. One can only
show that the delay term $F$ satisfies the inequality
\begin{equation}\label{sdd17-F1}
||F(\varphi)- F(\psi)||_{-1/2} \le  L_{F}\left(1+ {\rm Lip}_{[-r,0]}(A^{-\hf}\varphi)
 \right) |\varphi-\psi|_C
\end{equation}
for every $\va\in \cCL$ and $\psi\in C$. Using the terminology of
\cite{Mallet-Paret} we can call this mapping $F$  ``almost
Lipschitz'' from $C$ into $H_{-1/2}$, see also discussion in
\cite{Hartung-Krisztin-Walther-Wu-2006}.

\begin{remark}
 {\rm
  We can also include in \eqref{main-eq} a delay term
  $M(u_t)$ which is defined by a globally Lipschitz function
  from $C(-r,0; H_{1/2})$ into $H$. We will not pursue this generalization
  because our main goal is state-dependent delay models.
  }
\end{remark}

\section{Well-posedness}\label{sec:wp}

In this section we prove the existence and uniqueness theorem and study
properties of solutions. Then we use these results to construct the
corresponding evolution semigroup and describe its dynamical properties.
\par
We introduce the following definition.

\begin{definition}[Strong solution]
\label{D1}
{\rm
A vector-function
\begin{equation}\label{def-smooth}
u(t)\in C([-r,T];H)\cap
C([0,T];H_{1/2}) \cap L^2(0,T; H_1)
\end{equation}
is said to be a (strong)
 solution to the problem defined by (\ref{main-eq}) and
(\ref{sdd9-ic}) on $[0,T]$ if
\begin{itemize}
    \item[\textbf{(a)}] $u(\theta)=\varphi (\theta)$ for $\theta\in [-r,0]$;
    \item[\textbf{(b)}] $\forall v\in L^2(0,T; H) $ such that $\dot v\in
L^2(0,T; H_{-1}) $ and $v(T)=0$ we have that
\begin{multline}\label{sdd9-7}
-\int^T_0 \langle u(t),\dot v(t)\rangle \, dt  +
\int^T_0 \langle A u(t),  v(t)\rangle \, dt \\
+ \int^T_0 \langle F(u_t)+G(u(t)), v(t)\rangle \, dt
= \langle\varphi (0),v(0)\rangle + \int^T_0 \langle h, v(t)\rangle \, dt.
\end{multline}
\end{itemize}
}
\end{definition}
\begin{remark}\label{re:def-sol}
{ \rm Let $u(t)$ be a strong solution on an interval $[0,T]$ with
some $\va\in C$.
Then it follows from \eqref{def-smooth} and also from \eqref{G-lip} and
\eqref{lin-b-F} that
\[
F(u_t)+G(u(t))-h\in  L^\infty(0,T; H).
\]
This allows us to conclude from \eqref{def-smooth} and \eqref{sdd9-7}
that
\begin{equation}
\label{dot-sol1} \dot u(t)\in L^\infty(0,T; H_{-1/2})\cap L^2(0,T;
H).
\end{equation}
Moreover, the relation in \eqref{sdd9-7} implies that $u(t)$ satisfies
\eqref{main-eq} for almost all $t\in [0,T]$ as an equality in $H$.
We also note that relations \eqref{def-smooth} and \eqref{dot-sol1} yield
\begin{equation}
\label{dot-sol2}
u_t\in\cCL~~\mbox{for every}~~t\in [0,T]~~\mbox{and}~~
\max_{[0,T]} |u_t|_{\cCL}<+\infty
\end{equation}
for every strong solution $u$ with initial data $\va$ from the space $\cCL$
defined in \eqref{sdd9-8}.
}
\end{remark}

Our first result is  the following theorem
on the existence and uniqueness of solutions.
\smallskip

\begin{theorem}\label{th:wp}
 Let Assumption~\ref{as:basic} be in force.  Assume that $\varphi \in \cCL$,
see \eqref{sdd9-8}.
Then the initial-value problem defined by (\ref{main-eq}) and
(\ref{sdd9-ic}) has a unique
strong solution on any time interval $[0,T]$.
This solution possesses the property
\begin{equation}\label{deriv-est}
    \dot u(t)\in C(0,T; H_{-1/2})\cap L^2(0,T; H)
\end{equation}
and satisfies the estimate
\begin{equation}\label{b-sdd9-14fin}
||A^{-1/2}\dot u(t)||^2+||A^{1/2}u(t)||^2 + \int^t_0\left[  ||\dot
u(\tau)||^2+ ||Au(\tau)||^2\right] \, d\tau \le C_T(R)
\end{equation}
for all $t\in [0,T]$ and  $||A^{1/2}\varphi (0)||^2 +
  |\va |_C^2\le R^2$.
Moreover, for every two strong solutions $u^1$ and $u^2$ with initial data
$\va^1$ and $\va^2$ from $\cCL$ we have that
  \begin{equation}\label{lip-weak}
  \sup_{\tau\in [0,t]}  ||u^1(t)-u^2(t)||^2 +
  \int_0^t||A^{1/2}(u^1(\tau)-u^2(\tau))||^2\le  C_R(T)|\va^1-\va^2 |_C^2,
~~\forall\, t\in [0,T],
  \end{equation}
  for all $\va^i$ such that $|\va^i|_\cCL\le R$.
\end{theorem}

\medskip
\begin{proof}
 To prove the existence we use
 the  standard compactness  method  \cite{Lions} based on
 Galerkin approximations   \wrt the eigen-basis $\{e_k\}$ of the operator $A$
 (see Assumption 2.1 (A)).
\par
We define a  Galerkin approximate solution of order $m$ by the formula
\[
u^m=u^m(t,x)=\sum^m_{k=1} g_{k,m}(t) e_k,
\]
where the functions $g_{k,m}$ are defined on $[-r,T]$, absolutely continuous
on $[0,T]$  and  such that
 the following  equations are satisfied
 \begin{equation}\label{sdd9-12}
\left\{ \begin{array}{ll} &\langle \dot u^m+Au^m+F(u^m_{t})+G(u^m)-h,
e_k\rangle =0,~~t>0,\\ &
\langle u^m(\theta),e_k\rangle=\langle \varphi (\theta) , e_k\rangle,
\,\,\forall \theta\in [-r,0],~~\forall k=1,\ldots,m.\end{array}\right.
 \end{equation}
The equation in (\ref{sdd9-12}) is a system of (ordinary) differential
equations in $\mathbb{R}^m$ with a concentrated (discrete)
state-dependent delay for the unknown vector function $U(t)\equiv
(g_{1,m}(t), \ldots,g_{m,m}(t))$ (for the corresponding theory see
\cite{Walther_JDE-2003} and also the survey
\cite{Hartung-Krisztin-Walther-Wu-2006}).
\par
The condition $\varphi \in \cCL$ implies that the function
$U(\cdot)|_{[-r,0]}\equiv P_m \varphi (\cdot) $, which defines
initial data, is Lipschitz continuous as a function from $[-r,0]$ to
$\mathbb{R}^m$. Here $P_m$ is the orthogonal projection onto the
subspace ${\rm Span}\, \{ e_1,\ldots, e_m\}$.  Hence,
we can apply the theory of ODEs with discrete state-dependent delay
(see e.g.
\cite{Hartung-Krisztin-Walther-Wu-2006}) to get the \emph{local} existence
of solutions to (\ref{sdd9-12}).
\par
Next, we derive  an a priori  estimate which allows us to extend
solutions $u^m$ to (\ref{sdd9-12}) on an arbitrary  time interval
$[0,T]$.  We also use  it for the  compactness of the set of
approximate solutions.
\par
 We multiply
the first equation in (\ref{sdd9-12})   by $\lambda_k g_{k,m}$ and
sum for $k=1,\ldots,m$ to get
\begin{align*}
\frac{1}{2} \frac{d}{dt} ||A^{1/2}u^m(t)||^2 + ||Au^m(t)||^2 +
 \langle F(u^m_t)+G(u^m(t))-h ,Au^m(t)\rangle =0.
\end{align*}
Due to \eqref{lin-b-F} and \eqref{sdd17-G}
this implies that
\begin{align*}
\frac{d}{dt}\left[ ||A^{1/2}u^m(t)||^2 + \int_0^t ||Au^m(\tau)||^2 d\tau \right] &\le c \big[1+
  |u^m_t |_C^2 +||A^{1/2}u^m(t)||^2\big] \\
  & \le c_0  \big[1+
  |\va |_C^2\big] + c_1\max_{\tau\in [0,t]}||A^{1/2}u^m(\tau)||^2
\end{align*}
 Integrating the last inequality  we can easily  see that
the function
\[
\Psi(t)= \max_{\tau\in [0,t]} ||A^{1/2}u^m(\tau)||^2 + \int_0^t ||Au^m(\tau)||^2 d\tau
\]
satisfies the inequality
\begin{align*}
\Psi(t) &\le 2 ||A^{1/2}\va(0)||^2 +
  2t c_0  \big[1+
  |\va |_C^2\big] + 2c_1\int_0^t \Psi(\tau)  d\tau.
\end{align*}
Therefore Gronwall's lemma gives us the a priori estimate
\begin{equation}\label{sdd9-14}
||A^{1/2}u^m(t)||^2 + \int^t_0 ||Au^m(\tau)||^2 \, d\tau \le 2e^{at}\left[
||A^{1/2}\varphi (0)||^2 +b t  \big[1+
  |\va |_C^2\big] \right],
\end{equation}
for all $t$ from an existence interval,
 where $a$ and $b$   are positive constants. This a priori estimate allows us
 to extend  approximate solutions on every time interval $[0,T]$
 such that \eqref{sdd9-14} remains true for every $t\in [0,T]$.
\par
Now we establish  additional a priori bounds. Using \eqref{sdd9-14},
\eqref{G-lip} and \eqref{lin-b-F} from the first equation in (\ref{sdd9-12}) we obtain
that
\[
\| \dot u^m(t)+Au^m(t)\|\le \| F(u^m_{t})\|+\|G(u^m(t))\|+\|h\|\le C(R,T),~~t\in [0,T],
\]
provided $ ||A^{1/2}\varphi (0)||^2 + |\va |_C^2\le R^2$.
%
Thus by \eqref{sdd9-14} we obtain the estimate
\begin{equation}\label{sdd9-14fin}
||A^{1/2}u^m(t)||^2 + \int^t_0\left[  ||\dot u^m(t)||^2+ ||Au^m(\tau)||^2\right] \, d\tau \le C_T(R)
\end{equation}
for all $t\in [0,T]$ and  $||A^{1/2}\varphi (0)||^2 +
  |\va |_C^2\le R^2$. It  also follows from  (\ref{sdd9-12})  that
\begin{equation}\label{sdd9-14fin2}
\sup_{t\in [0,T]}||A^{-1/2}\dot u^m(t)||^2 \le C_T(R).
\end{equation}
Thus
$$
\{ u^m \}^\infty_{m=1} \hbox{ is a bounded set in } W_1\equiv
L^\infty(0,T;H_{1/2})\cap L^2(0,T;D(A)).
$$
and
$$
\{ \dot u^m \}^\infty_{m=1} \hbox{ is a bounded set in } W_2\equiv
L^\infty(0,T;H_{-1/2})\cap L^2(0,T; H).
$$
Hence, there exist a subsequence $\{ (u^k;\dot u^k ) \}$ and an
element  $(u;\dot u )\in Z_1\equiv W_1\times W_2$ such that
\begin{equation*}
\{ (u^k;\dot u^k ) \} \hbox{ *-weak converges to } (u;\dot u )
\hbox{ in } Z_1.
\end{equation*}
By the Aubin-Dubinski theorem \cite[Corollary 4]{Simon-AMPA-1987} we also have
\[
 u^m \to u  \hbox{ in } C(0,T;H_{1/2-\delta}))\cap L^2(0,T;H_{1-\delta}).
\]
Now the proof that any *-weak limit $u(t)$ is a solution is
standard. To make the limit transition in the nonlinear terms $F$
and $G$ we use relation \eqref{sdd17-F1} and Assumption
\ref{as:basic}\textbf{(Gb)}.

\par
 The property $u(t) \in C([0,T]; H_{1/2}))$ follows from the
 well-known continuous embedding (see also \cite[Theorem 1.3.1]{Lions-Magenes-book} 
or \cite[Proposition 1.2]{showalter}):
 \[
  \{ u\in L^2(0,T;H_1) : \dot u\in L^2(0,T; H)\} \subset C([0,T];H_{1/2}).
 \]
 The continuity of $\dot u$ in $H_{-1/2}$ follows from equation \eqref{main-eq}
 and from continuity of $u$ in $H_{1/2}$.
Thus the existence of strong solutions is proved. It is easy to see from \eqref{sdd9-14fin}
and \eqref{sdd9-14fin2}
that the strong solution constructed satisfies  \eqref{b-sdd9-14fin}.

\smallskip\par
Now we use this fact to prove the uniqueness.
\par
 Let $u^1$ and $u^2$ be two solutions
(at this point we do not assume that they have the same initial data).
Then the difference $z=u^1-u^2\in C([0,T];H_{1/2})\cap L^2(0,T;H_1)$ is a strong solution to
 the linear parabolic type (non-delay) equation
\begin{equation}\label{main-eq-dif}
\dot z(t) +A z(t) =f(t),~~t>0, ~~\mbox{with}~~ f(t)\equiv F(u^2_t)-F(u^1_t)+G(u^2(t))-G(u^1(t)).
\end{equation}
By Remark~\ref{re:def-sol}
$f\in L^\infty(0,T;H)$. From \eqref{G-lip} and \eqref{sdd17-F1}
using \eqref{dot-sol2} we also have that
\[
\|G(u^2(t))-G(u^1(t))\|\le L(\vr)\|z(t)\|_{1/2},~~ t\in [0,T],
\]
and
\[
\|A^{-1/2}(F(u^2_t)-F(u^1_t))\| \le
L_F(1+\vr)|z_t|_{C},~~ t\in [0,T],
\]
for every $\vr \ge \max_{[0,T]}\left\{ |u^1_t|_{\cCL}+
|u^2_t|_{\cCL}\right\}$. Therefore
\begin{align*}
|\langle f(t), z(t)\rangle| \le & L_F(1+\vr)|z_t|_{C}
\|z(t)\|_{1/2}+L(\vr)\|z(t)\|_{1/2} \|z(t)\| \\
\le &\hf \|z(t)\|^2_{1/2} + C(\vr)|z_t|_{C}.
\end{align*}
Thus using the standard
 multiplier $z$ in \eqref{main-eq-dif}
 we obtain that
\begin{align*}
 \frac{d}{dt} ||z(t)||^2 + ||A^{1/2}z||^2 &\le
 C(\vr)\| z_t\|_C^2 \le  C(\vr) \left[|\va^1-\va^2 |_C^2 + \sup_{\tau\in [0,t]} \| z(\tau)\|_C^2
\right]
\end{align*}
for every $\vr \ge \max_{[0,T]}\left\{ |u^1_t|_{\cCL}+
|u^2_t|_{\cCL}\right\}$. Applying Gronwall's lemma we obtain
  \begin{equation}\label{lip-weak2}
  \sup_{\tau\in [0,t]}  ||u^1(t)-u^2(t)||^2 +
  \int_0^t||A^{1/2}(u^1(\tau)-u^2(\tau))||^2\le  C(\vr)|\va^1-\va^2 |_C^2,
~~\forall\, t\in [0,T],
  \end{equation}
  This implies  uniqueness of strong solutions.
  \par
  As a by-product the uniqueness yields
  that {\it any} strong solution satisfies \eqref{b-sdd9-14fin}.
  Therefore we can apply \eqref{lip-weak2}
  with $\vr>R+C_T(R)$ to obtain \eqref{lip-weak}.
  \par
 Thus the proof of Theorem \ref{th:wp} is complete.
 \end{proof}
Theorem \ref{th:wp} allows us to define
an evolution semigroup $S_t$ on the space $\cCL$ (see (\ref{sdd9-8})) by the formula
\begin{equation}\label{sdd10-26}
S_t \varphi \equiv u_t,\quad t\ge 0,
\end{equation}
where $u(t)$ is the unique solution to the problem (\ref{main-eq})
and (\ref{sdd9-ic}). We note that \eqref{lip-weak} implies that
$S_t$ is \textit{almost } locally  Lipschiz on $C$, i.e.,
\[
|S_t\va^1-S_t\va^2|_{C}\le C_R(T)
|\va^1-\va^2|_{C}~~\mbox{for every}~~\va^i\in \cCL,~~|\va^i|_{\cCL}\le
R,~~ t\in [0,T]
\]
However, it seems that a similar bound is not true in the
space $\cCL$. We can only guarantee
 that $\va\mapsto S_t \varphi$ is a continuous mapping
 on $\cCL$ for all $t> r$.
 Moreover,
  the following assertion shows that the mapping
   $\va\mapsto S_t \varphi$  is even $\hf$-H\"{o}lder
 on $\cCL$ \wrt $\va$ when $t>r$.

  \begin{proposition}[Dependence on initial data in the space $\cCL$]\label{pr:ini-data}
   Assume that the hypotheses of Theorem \ref{th:wp} are in force.
    Let $u^1$ and $u^2$ be two solutions on $[0,T]$
    with initial data $\va^1$ and $\va^2$ from $\cCL$.
  Then the difference  $z=u^1-u^2$ satisfies the estimate
\begin{equation}\label{b-sdd9-14dif}
(t-r)\left[||A^{-1/2}\dot z(t)||^2+||A^{1/2}z(t)||^2\right] + \int^t_{r}(\tau-r)\left[  ||\dot z(\tau)||^2+ ||Az(\tau)||^2\right] \, d\tau   \le C_T(R) |\va^1-\va^2 |_C
\end{equation}
for all $t\in [r,T]$ and  for all initial data $\va^i$ such that
$|\va^i|_\cCL \le R$.
 This implies that for  every $t> r$ the evolution semigroup $S_t$ is $\hf$-H\"older continuous in the norm of $\cCL$.
 In the case when $t\in (0,r]$ we can  guarantee the closeness of the evolution operator
 $S_t$ only. This means\footnote{We refer to the Appendix
 for a discussion of closed evolutions. Here we only mention that
 any continuous mapping is closed
 and a mapping can be closed  but not continuous, see examples in \cite{PataZelik-cls}
 and also in \cite[Sect.1.1]{Chueshov-2014_book}.
}
  (see, e.g., \cite{PataZelik-cls}) that the properties $\va_n\to\va$ and
$S_t\va_n\to\psi$ in  the norm of $\cCL$ as $n\to\infty$ imply
that $S_t\va =\psi$.
  \end{proposition}
  \begin{proof}
  Multiplying \eqref{main-eq-dif} by  $Az$ and using \eqref{b-sdd9-14fin} and \eqref{G-lip} we obtain that
\begin{align*}
 \frac{d}{dt} ||A^{1/2}z(t)||^2 + ||A z(t)||^2 &\le
C\| F(u^2_t)-F(u^1_t)\|^2  +C_R(T) ||A^{1/2}z(t)||^2,\, t>0.
\end{align*}
From \eqref{b-sdd9-14fin}, \eqref{F-lip} and \eqref{eta-lip} we also have that
 \begin{align}\label{F-lip-l2}
\| F(u^2_t)-F(u^1_t)\|^2 & \le L_F \left|
\int_{t-\eta(u^1_t)}^{t-\eta(u^2_t)} \|\dot u^2(\xi)\|d\xi + |
u^2_t- u^1_t|_C\right|^2 \notag
\\
&\le 2 L_F\left[ \left|\eta(u^1_t)-\eta(u^2_t)\right|
 \int_{t-r}^{t}
\|\dot u^2(\xi)\|^2d\xi + | u^2_t- u^1_t|_C^2\right] \le  C_T(R)  |
u^2_t- u^1_t|_C
 \end{align}
 for every $t\ge r$.
 Therefore
 \begin{align*}
 \frac{d}{dt} ||A^{1/2}z(t)||^2 + ||A z(t)||^2 &\le
 C_T(R)\left[ \max_{[0,t]}\| z(s)\|^2  +
 ||A^{1/2}z(t)|^2\right]^{1/2},
 ~~t\ge r.
\end{align*}
Integrating  over interval $[\tau, t]$ with $\tau\ge r$ and
 using \eqref{lip-weak} we obtain that
 \begin{equation} \label{b2-sdd-holder}
 ||A^{1/2}z(t)||^2 + \int_\tau^t||A z(\xi)||^2 d\xi \le ||A^{1/2}z(\tau)||^2+
 C_T(R) |\va^1-\va^2 |_C ,~~t\ge\tau\ge r.
\end{equation}
Now we integrate \eqref{b2-sdd-holder} with respect to
$\tau$ over  $[r, t]$, change the order of integration, and use \eqref{lip-weak}
to get
$$
(t-r)||A^{1/2}z(t)||^2 + \int^t_{r} (\xi-r) ||Az(\xi)||^2 \, d\xi   \le
C_T(R) |\va^1-\va^2 |_C ,~~t\ge r.
$$
Using the expression for $\dot z$ from
\eqref{main-eq-dif} and also the bounds in \eqref{lip-weak} and
\eqref{F-lip-l2} we have that
$$
||\dot z(t)+Az(t)||^2 +||A^{-1/2}\dot z(t)||^2 \le C_T(R)
\left[ \|A^{1/2} z(t)\|^2+ |\va^1-\va^2|_C\right],~~t\ge r.
$$
This implies \eqref{b-sdd9-14dif}.
\par
The $\hf$-H\"older continuity of the evolution
semigroup $S_t$ in the norm of $\cCL$ follows from \eqref{b-sdd9-14dif}.
\par
 The closedness of $S_t$ for $t\in (0,r]$ easily follows from
\eqref{lip-weak}.
 \end{proof}

 \begin{remark}\label{re:L0}
 {\rm As it follows from \eqref{F-lip-l2} we can obtain a $\hf$-H\"{o}lder continuity
  relation like
 \eqref{b-sdd9-14dif} \emph{for all} $t\ge 0$ if we assume
 in addition that one of initial data
 $\va^i$ possesses the property $\dot\va^i\in L_2(-r,0;H)$. In this case
 the argument above leads to the relation
 \begin{multline}\label{b-sdd9-14dif+0}
||A^{-1/2}\dot z(t)||^2+||A^{1/2}z(t)||^2 + \int^t_0\left[
||\dot z(t)||^2+ ||Az(\tau)||^2\right] \, d\tau  \\ \le C_T(R)\left[ \|A^{1/2}(\va^1(0)-\va^2(0))\|+
|\va^1-\va^2 |_C \right]
\end{multline}
for all $t\in [0,T]$ and  for all initial data $\va^i$ such that
$|\va^i|_\cCL +|\dot\va^i|_{ L^2(-r,0;H)} \le R$. Moreover,   one
can also see that the set
\begin{equation}\label{L0-def}
    \cCL_0=\left\{ \va\in \cCL\, :\; \dot\va\in L^2(-r,0;H)\right\}
\end{equation}
is forward invariant \wrt $S_t$.
 Thus $\va\mapsto S_t \varphi$ is a $\hf$-H\"{o}lder continuous mapping
 for each $t\ge 0$ on the Banach space $\cCL_0$
 endowed with the norm
$|\va|_{\cCL_0}=|\va|_\cCL +|\dot\va|_{ L^2(-r,0;H)}$.
Hence the dynamical (in the classical sense, see, e.g.,
 \cite{Babin-Vishik,Chueshov_Acta-1999_book,Temam_book}) system $(\cCL_0,S_t)$
 arises. However we prefer to avoid property $\dot\va\in L_2(-r,0;H)$
in the description of the phase space. The point is that our goal is
long-time dynamics and it is well-known (see, e.g.,
\cite{Babin-Vishik,Chueshov_Acta-1999_book,Temam_book}) that the
existence of limiting objects requires some compactness properties.
Unfortunately we cannot guarantee these properties in  the space
$\cCL_0$ without serious restrictions concerning the delay term.
This is why we prefer to use the observation made in
\cite{PataZelik-cls} concerning closed evolutions.
 }
 \end{remark}
 \begin{remark}\label{re:cont-in-t}
 {\rm
A similar problem  as above we have with \emph{time} continuity of
evolution operator $S_t$. It is clear from \eqref{def-smooth} and
\eqref{deriv-est} that $t\mapsto S_t \varphi$  is continuous for
every $\va\in\cCL$  when  $t> r$. To guarantee the continuity $
t\mapsto S_t \varphi$ \emph{for all} $t\ge 0$ we need make further
restriction\footnote{We refer to some discussion in
\cite{Rezounenko_NA-2010,Rezounenko-Zagalak-DCDS-2013} for the
related PDE models.} on initial data. The main restriction is a
compatibility condition at time $t=0$. To describe    this condition
we introduce the following (complete)  metric space
\begin{equation}\label{sdd10-36}
X \equiv \left\{ \varphi \in C^1([-r,0];H_{-1/2})\cap C([-r,0];H) \left|
\begin{array}{l}
 \varphi (0)\in H_{1/2};   \\
   \dot \varphi (0) + A \varphi (0)
+ F(\varphi)+G(\va(0))=0
\end{array} \right.\right\}
\end{equation}
Here the compatibility condition $\dot \varphi (0) + A \varphi (0)+
F(\varphi)+ G(\varphi (0))=0$ is understood as an equality in $H_{-1/2}$.
The distance in $X$ is given by the relation
\begin{equation}\label{x-norm}
{\rm dist}_X(\va,\psi)=\max_{[-r,0]}\left\{ ||A^{-1/2}(\dot
\va(\theta) -\dot\psi(\theta))||+ ||
\va(\theta) -\psi(\theta)||\right\} +||A^{1/2}(\va(0) -\psi(0))||.
\end{equation}
One can see that $X$ is a closed subset in  the Banach space $\cCL$
and the topology generated by the  metric ${\rm dist}_X$ coincides
with the induced topology of $\cCL$ (see \ref{sdd9-9}).
}
 \end{remark}
\par
In the following assertion we collect several dynamical properties
of the evolution semigroup $S_t$ which are direct consequences
of Theorem~\ref{th:wp} and Proposition \ref{pr:ini-data}
and Remark \ref{re:cont-in-t}.
\begin{proposition}\label{pr:closed-ev}
Under the conditions of Theorem~\ref{th:wp} problem \eqref{main-eq}
generates an evolution semigroup $S_t$ of closed mappings  on $\cCL$
such that
\begin{enumerate}
  \item[\textbf{(a)}] $S_t\cCL\subset X$ for every $t\ge r$ and the set
  $S_tB$ is bounded in $X$ for each $t\ge r$ when $B$ is bounded in the space $\cCL$;
 \item[\textbf{(b)}] the set $X$ is forward invariant: $S_tX\subset X$;
  \item[\textbf{(c)}]
  the mapping
$\va\mapsto S_t \varphi$ is a $\hf$-H\"{o}lder continuous on $\cCL$
 (and hence on $X$) for all $t>r$;
 \item[\textbf{(d)}] the trajectories $t\mapsto S_t\va$  are continuous
 for $t> r$ and $\va\in \cCL$. If $\va\in X$, then these
 trajectories
 are continuous for all $t\ge 0$.
\end{enumerate}
\end{proposition}

\section{Long time dynamics}
This section is central for the whole paper. Here we study long-time
dynamics of the delay model generated by  \eqref{main-eq}  and
\eqref{sdd9-ic}. The main result stated below in
Theorem~\ref{th:attr} deals with finite-dimensional global and
exponential attractors. We refer to the Appendix for the
corresponding definitions and the auxiliary facts which we use in
our argument.

We first impose the standard hypotheses (see, e.g.,
\cite{Temam_book}) concerning the nonlinear
(non-delayed)  sink/source term $G$.
\begin{assumption}\label{as:G2}
  {\rm
 The nonlinear mapping $G: H_{1/2} \to H$ has the form
 $G(u) = \Pi^\prime(u)$. Here
 $\Pi(u)=\Pi_0(u) + \Pi_1(u)$, where
$\Pi_0(u)\ge 0$  is bounded on bounded sets in
$D(A^{1/2})$ and $\Pi_1(u)$ satisfies the property
\begin{equation}\label{sdd-2nd-08}
\forall\,\eta>0\; \exists\, C_\eta>0:~~
 |\Pi_1(u)| \le \eta \left( || A^{1\over 2}u||^2 + \Pi_0(u)\right) + C_\eta, \qquad u \in H_{1/2}. 
\end{equation}
Moreover, we assume that
 \par\noindent
{\bf (a)}
 there are constants $\eta\in [0,1), c_4,c_5>0$
such that
\begin{equation}\label{sdd-2nd-19}
- \langle u,G(u)\rangle\le \eta ||A^{1/2}u||^2 -c_4 \Pi_0(u)
+c_5,\quad u\in H_{1/2}; 
\end{equation}
{\bf (b)}
for every $\widetilde{\eta}>0$ there exists $C_{\widetilde{\eta}}>0$
such that
\begin{equation}\label{sdd-2nd-20}
||u||^2\le C_{\widetilde{\eta}} +\widetilde{\eta}\left( ||A^{1\over
2}u||^2 +\Pi_0(u) \right), \quad u\in H_{1/2}. 
\end{equation}
}
\end{assumption}
In the case of parabolic models like \eqref{sdd9-1}  examples of
functions $g(u)$ such that the corresponding Nemytskii operator
satisfies Assumptions \ref{as:basic}(G) and \ref{as:G2} can be found
in \cite{Babin-Vishik} and \cite{Temam_book}. The simplest one is
$g(u)=u^3+a_1u^2+a_2 u$ with arbitrary $a_1,a_2\in \R$ in the case
when $\Omega$ is a 3D domain.
\par
Our main result is the following assertion.
\begin{theorem}\label{th:attr}
  Let Assumptions~\ref{as:basic}  and \ref{as:G2} be in force.
  Suppose that $S_t$ is the evolution semigroup generated
 in $\cCL$ by  \eqref{main-eq}  and \eqref{sdd9-ic}. Then
there exists $\ell_0>0$ such that
 this semigroup possesses a compact connected global
 attractor $\fA$
provided $m_F r< \ell_0$, where $r$ is the delay time and
$m_F$ is the linear growth constant for $F_0$ in $H$
defined by the relation
\begin{equation}\label{mF-def}
m_F=\limsup_{\|u\|\to+\infty}\frac{\|F_0(u)\|}{\|u\|}.
\end{equation}
  Moreover,  for every $0<\beta\le 1$ and $\al<\min\{\beta,1/2\}$
  this attractor belongs to the set
  \begin{equation}\label{dis-comp}
    D_{\al,\beta}^R=\left\{ \varphi\in X\, \left|\begin{array}{r}
                                                 |A^{1-\beta}\va |_C+  |A^{-\beta}\dot\va|_C
                                                  +{\rm Hold}_\al(A^{1-\beta}\va)
                                                 +{\rm Hold}_\al(A^{-\beta}\dot\va)
                                                 \\ \displaystyle
                                                 +\left[ \int_{-r}^0\left( \|A^{1/2}\va (\theta)\|^2+ \|\dot\va(\theta)\|^2\right) d\theta\right]^{1/2}
                                                 \le R
                                               \end{array}\right. \right\}
  \end{equation}
 for some $R=R(\al,\beta)$, where the H\"{o}lder seminorm ${\rm Hold}_\al(\psi)$
 is given by
 \[
 {\rm Hold}_\al(\psi)=\sup\left\{\frac{\|\psi(t_1)-\psi(t_2)\|}{|t_1-t_2|^\al}\, :\;
 t_1\neq t_2, ~t_1,t_2\in [-r,0]\right\}.
 \]
 Assume in addition that there exist $\ga,\delta>0$ such that
 {\bf (a)} the mapping $F_0$ is globally Lipschitz from $H_{-\ga}$ into $H_{-1/2+\delta}$,
 i.e.,
 \begin{equation}\label{F0-fd}
    \|F_0(u)-F_0(v)\|_{-1/2+\delta}\le c   \| u-v\|_{-\ga},~~ u,v\in H_{-\ga};
 \end{equation}
 and {\bf (b)} the mapping $G$ is 
 globally Lipschitz from $H_{-\ga}$ into $H_{-1/2+\delta}$,
 i.e.,
 \begin{equation}\label{G-fd}
    \|G(u)-G(v)\|_{-1/2+\delta}\le c(R)   \| u-v\|_{1/2-\ga},~~ u,v\in H_{1/2-\ga},~~
    \|u\|, \|v\|\le R.
 \end{equation}
 Then
 \begin{enumerate}
 \item[\textbf{(A)}] The global attractor $\fA$ has
 finite fractal dimension.
  \item[\textbf{(B)}] There exists a fractal exponential attractor $\fA_{\rm exp}$.
 \end{enumerate}
\end{theorem}

We devote the remaining subsections to the proof of Theorem~\ref{th:attr}.

\subsection{Existence of a global attractor}
To prove the existence of a global attractor it is sufficient to
show that the evolution operator possesses a compact absorbing set.
In this case we can apply the standard existence result in the form
given in \cite{PataZelik-cls} for closed semigroups
(see the Appendix for more details).
\par
We start with the existence of a bounded absorbing set.

\begin{proposition}[Bounded dissipativity]\label{pr:diss}
Assume that $u(t)$ solves \eqref{main-eq} and \eqref{sdd9-ic} with
$\va\in\cCL$. Then one can find  $\ell_0>0$ such that for every delay time $r$
such that
 $m_F r< \ell_0$ 
the following property holds: there exists $R_*$ such that  for
every bounded set $B$ in $\cCL$ there is $t_B$ such that
\begin{equation}\label{dis-est-str}
||A^{-1/2}\dot u(t)||^2+||A^{1/2}u(t)||^2 + \int^{t+1}_{t}\left[  ||\dot u(\tau)||^2+ ||Au(\tau)||^2\right] \, d\tau   \le R^2_*
\end{equation}
for all $t\ge t_B$ and  for all initial data $\va\in B$. This yields
that  the evolution semigroup $S_t$ is dissipative on both $\cCL$
and
$X$ provided $m_F r< \ell_0$.
\end{proposition}
\begin{proof}
 We use the Lyapunov method  to get the
result. For this we consider the following functional
\begin{equation*}
\widetilde{V}(t)\equiv \hf\left[ \|u(t)\|^2 +\| A^{1/2} u(t)\|^2\right] + \Pi(u(t))+
 \frac{\mu}r \int_0^r \left\{ \int_{t-s}^t ||\dot u(\xi)||^2
d\, \xi \right\} \, ds.
\end{equation*}
defined on strong solutions $u(t)$ for $t\ge r$.
The positive parameter  $\mu$ will be chosen later.
We note that
the main idea behind inclusion of an additional delay term in $\widetilde{V}$  is to
find a compensator for the delay term in \eqref{main-eq}. This idea
was already applied in \cite{ChuRez2013-sdd} for second oder in time models with state-dependent term, see also \cite[p.480]{Chueshov-Lasiecka-2010_book} and
\cite{CLW-delay}  for the case of a flow-plate interaction model
which contains a linear constant delay term with the critical spatial
regularity. The corresponding compensator is model-dependent.
\par
One can see from \eqref{sdd-2nd-08} that there
is $0<c_0<1$ and $c,c_1>0$
such that
\begin{equation}\label{sdd-2nd-23a}
c_0 \left[\| A^{1/2} u(t)\|^2 + \Pi_0(u(t))\right] -c\le
\widetilde{V}(t)\le  c_1 \left[\| A^{1/2} u(t)\|^2 + \Pi_0(u(t))\right] +  \mu \int_0^r ||\dot u(t-\xi)||^2 d\, \xi +c.
\end{equation}
\par
We consider the time derivative of $\widetilde{V}$ along a solution.
One can easily check that
\begin{align*}
\frac{d}{dt} \widetilde{V}(t) = &  \langle  u(t),\dot u(t)\rangle +  \langle  Au(t),\dot u(t) \rangle
+ \langle G(u(t)), \dot u(t) \rangle +\frac{\mu}r \int_0^r
\left\{||\dot u(t)||^2 - ||\dot u(t-s)||^2 \right\} d\, s \\
=&
\langle \dot u(t)+ Au(t) +G(u(t)), \dot u(t) \rangle - \langle \dot
u(t),\dot u(t)\rangle + \langle u(t),\dot u(t)\rangle + \mu ||\dot
u(t)||^2  \\ &
 -\frac{\mu}r \int_0^r ||\dot u(t-\xi)||^2 d\, \xi \\
= & - \langle F(u_t) - h,\dot u(t)\rangle - (1 -\mu) ||\dot u(t)||^2
 -\frac{\mu}r \int_0^r ||\dot u(t-\xi)||^2 d\, \xi \\
 &- ||A^{1/2}u(t)||^2 -\langle F(u_t)+G(u(t))-h,u(t) \rangle.
\end{align*}
The last terms are due to \eqref{main-eq}:
$$ \langle  u(t),\dot u(t)\rangle  = -\langle  Au(t), u(t) \rangle -
\langle F(u_t)+G(u(t)) - h, u(t)\rangle.
$$
%
By the definition of $m_F$ in \eqref{mF-def}
for any number  $M_F$  greater than $m_F$
we can find $C(M_F)$ such that
\[
||F(u_t)|| \le ||F_0(u(t-\eta(u_t)))|| \le
M_F ||u(t-\eta(u_t))||+ C(M_F).
\]
Therefore
\begin{align*}
||F(u_t)|| \le & M_F
||u(t-\eta(u_t))-u(t)||+ M_F||u(t)||+ C(M_F) \\
 = & M_F \left\|
\int^t_{t-\eta(u_t)} \dot u(\theta) d\, \theta \right\| +
M_F||u(t)||+ C(M_F),
\end{align*}
and thus
\[
\|F(u_t)\|\le M_F\cdot \left[\|u(t)\|+ \int_0^r ||\dot u(t-\xi)||
d\, \xi \right]+ C(M_F), ~~t\ge r.
\]
Since
$$
\int_0^r ||\dot u(t-\xi)|| d\, \xi \le r^{1/2} \left( \int_0^r ||\dot u(t-\xi)||^2 d\, \xi \right)^{1/2},
$$
we have that
\begin{align*}
|\langle F(u_t)-h,\dot u(t)\rangle| \le & \frac12  \|\dot u(t)\|^2 +c_0\|h\|^2
+c_1M_F^2\|u(t)\|^2 \\ &+ c_2 M_F^2 r\int_0^r ||\dot u(t-\xi)||^2 d\, \xi +
C(M_F), ~~t\ge r. \notag
 \end{align*}
In a similar way  we also have that
\begin{align*}
|\langle F(u_t)-h, u(t)\rangle)| \le c_1 M_F^2  
r  \int^r_0 ||\dot u(t-\xi)||^2 d\,
\xi + C(M_F)(1+ ||u(t)||^2).
\end{align*}
Thus
\begin{multline*}
|\langle F(u_t)-h,\dot u(t)\rangle| +
|\langle F(u_t)-h, u(t)\rangle)| \le
\\
\frac12  \|\dot u(t)\|^2 +
c_0 M_F^2
r  \int^r_0 ||\dot u(t-\xi)||^2 d\,
\xi + c_1(M_F)(1+ ||u(t)||^2).
\end{multline*}
The relations in  \eqref{sdd-2nd-19} and  (\ref{sdd-2nd-20}) with
small enough $\widetilde{\eta}>0$ (and $\eta \in [0,1)$) yield
\[
 c_1(M_F)(1+ ||u||^2)  -||A^{1\over 2}u||^2 -(u,G(u))\le -  a_0\left[||A^{1/2}u||^2 + \Pi_0(u)\right]
 + a_1(M_F)
\]
for  some $a_i>0$ with $a_0$ independent of $M_F$.
Thus it follows from  the relations above that
\begin{align*}
\frac{d}{dt} \widetilde{V}(t) \le & 
-\left(\frac12 -\mu \right) ||\dot u(t)||^2
 \\ &
-  a_0 \left[||A^{1/2}u||^2 + \Pi_0(u)\right] + a_1(M_F)   + \left[
-\frac{\mu}r + a_2 M_F^2 
r\right] \int_0^r ||\dot u(t-\xi)||^2 d\, \xi
\end{align*}
for some $a_i$. Thus using the right inequality in
\eqref{sdd-2nd-23a} we arrive at the relation
\begin{align}\label{sdd-2nd-25b}
\frac{d}{dt} \widetilde{V}(t) +\ga  \widetilde{V}(t)  \le &
-\left( \frac12 -\mu \right) ||\dot u(t)||^2    -
\ ( a_0-\ga c_1)  \left[||A^{1/2}u||^2 + \Pi_0(u)\right]
\notag \\ & + \left[ -\frac{\mu}r +\mu\ga+ a_2 M_F^2 
r \right] \int_0^h
||\dot u(t-\xi)||^2 d\, \xi +a_1(M_F).
\end{align}
Therefore taking $\mu=1/4$ and fixing  $\ga\le a_0c_1^{-1}$  we obtain
that
\begin{align}\label{p-sdd-2nd-25b}
\frac{d}{dt} \widetilde{V}(t) +\ga  \widetilde{V}(t)  +\frac14 ||\dot u(t)||^2 \le C,
~~t\ge r,
\end{align}
 provided  $\ga r + 4a_2M_F^2 r^2\le 1$.
 Thus under the condition
$4a_2m_F^2 r^2< 1$ we can choose $\ga\in (0,a_0c_1^{-1}]$ and $M_F>m_F$
such that \eqref{p-sdd-2nd-25b} holds.
In particular we have that
 \[
 \frac{d}{dt} \widetilde{V}(t) +\ga
\widetilde{V}(t)\le C, ~~t\ge r,
\]
 which implies
\begin{align}\label{sdd-2nd-25d}
 \widetilde{V}(t) \le\widetilde{V}(r) e^{-\ga ( t-r)} +
\frac{C}{\ga}(1-e^{-\ga  (t-r)}), ~~t\ge r,
\end{align}
when $m_F r< \ell_0$. Using \eqref{sdd-2nd-23a} and \eqref{b-sdd9-14fin}
we can conclude that $|\widetilde{V}(r)|\le C_B$
for
all initial data from a bounded set $B$ in $\cCL$.
 Hence (see \eqref{main-eq})
 there exists $R$ such that for
every initial data from a bounded set $B$ in $\cCL$ we have that
\[
\|A^{1/2}u(t)\|+\|A^{-1/2}\dot u(t)\|+\|\dot u(t)+ A u(t)\|\le R  ~~~\mbox{for all}~~t\ge t_B.
\]
Moreover, it follows from \eqref{p-sdd-2nd-25b} that
\[
\int_t^{t+1}\|\dot u(\tau)\|^2 d t \le C_R  ~~~\mbox{for all}~~t\ge
t_B.
\]
To get this one should multiply \eqref{p-sdd-2nd-25b} by
$e^{\gamma t}$, integrate over $[t, t+1]$ and multiply
 by $e^{-\gamma t}$. Then ultimate boundedness of $\widetilde{V}(t)$ (see
 \eqref{sdd-2nd-25d}) and the relation
 $ 1\le e^{\gamma (\tau -t)}$ for $\tau \ge t$ give the last estimate.
\par
These relations  imply \eqref{dis-est-str} and allow us to complete
 the proof of Proposition~\ref{pr:diss}.
\end{proof}

\begin{remark}\label{F0-bounded}
 {\rm If the mapping $F_0$ has sublinear growth in $H$,
 i.e., there exists $\beta<1$ such that
 \[
 \|F_0(u)\|\le c_1+c_2\|u\|^{\beta},~~u\in H,
 \]
  then the linear growth parameter $m_F$ given by \eqref{mF-def} is zero. Thus in this case we
  have no restrictions concerning $r$ in the statement of
  Proposition~\ref{pr:diss}.
  In particular, this is true in the
   case of bounded mappings $F_0$. Moreover, in the latter case
   the argument can be simplified substantially (we can use a Lyapunov type
   function without delay terms).
    For more details  we refer
    to \cite{Rezounenko_NA-2010,Rezounenko-Zagalak-DCDS-2013}.
  }
\end{remark}

We use Proposition~\ref{pr:diss} to obtain the following assertion
which means that the evolution semigroup $S_t$ is (ultimately) compact.

\begin{proposition}[Compact dissipativity]\label{pr:dis-comp}
As in Proposition~\ref{pr:diss} we assume that $m_F r<\ell_0$. Then
  the evolution operator $S_t$ possesses a compact absorbing set.
   More precisely,
  for every $0<\beta\le 1$ and $\al<\min\{\beta,1/2\}$
   the set   $D_{\al,\beta}^R$ given by \eqref{dis-comp}
  is absorbing for some $R$. This set $D_{\al,\beta}^R$ is compact in
  $X$ provided $0<\al<\beta\le1/2$.
\end{proposition}
\begin{proof}
We first note that the
compactness of  $D_{\al,\beta}^R$ in $X\subset\cCL$ for
$0<\al<\beta\le1/2$
follows from Arzel\`{a}-Ascoli theorem in Banach spaces (see, e.g.,
\cite{Simon-AMPA-1987}).
\par
Now we show that $D_{\al,\beta}^R$ is absorbing.
\par
  Using the mild form of the problem and
  also the bound in \eqref{lin-b-F}  one can also show that
\begin{equation}\label{dis+delta}
 \|A^{1-\delta}u(t)\|+\|A^{-\delta}\dot u(t)\|\le C_R(\delta)
 ~~~\mbox{for all}~~t\ge t_B,
\end{equation}
for every $\delta>0$, where $u(t)$ is a solution possessing property
\eqref{dis-est-str}.
\par
Now we consider  the difference $u(t_1)-u(t_2)$
with $t_1>t_2$.
Namely, using the mild form we obtain
\begin{align*}
    ||A^{1-\beta}( u(t_1)-u(t_2))|| \le & ||A^{1-\beta} (e^{-A(t_1-t_2)}-1)u(t_2) ||
    \\ &
     + \int^{t_1}_{t_2} ||A^{1-\beta} e^{-A(t-\tau)}||\cdot
     (||F(u_\tau)||+ \|G(u(\tau))-h\|)\, d\tau.
\end{align*}
Since (see \cite[Theorem 1.4.3, p.26]{Henry_Springer-1981_book} for related
facts) 
\[
||A^{-\alpha}(1- e^{-At})|| \le t^\alpha~~\mbox{and} ~~||A^{\alpha}
e^{-At}|| \le \left(\frac{\alpha}{t}\right)^\alpha  e^{-\alpha}
\]
for all $t>0$ and $0\le\al\le1$, we obtain
\begin{align*}
    ||A^{1-\beta}( u(t_1)-u(t_2))|| \le & |t_1-t_2|^\al ||A^{1-\beta+\al}u(t_2) ||
    \\ &
     + c_\beta\int^{t_1}_{t_2} \frac1{|t_1-\tau|^{1-\beta}}\big[ C_{R_*}+c |u_\tau|_C\big]
     \, d\tau.
\end{align*}
for $t\ge t_B$. Thus for every $0<\al<\beta\le 1$ we have
\begin{equation}\label{holder1}
    ||A^{1-\beta}( u(t_1)-u(t_2))|| \le C_{R_*}|t_1-t_2|^\al~~\mbox{for all}~~ t_i\ge t_B,
    |t_1-t_2|\le 1.
\end{equation}
Similarly to \eqref{F-lip-l2} using \eqref{holder1} with $\beta=1$
and $\al =1/2$ we have that
\begin{align*}
\| F(u_{t_1})-F(u_{t_2})\| & \le L_F \left| \int_{t_1-\eta(u_{t_1})}^{t_2-\eta(u_{t_2})}
\|\dot u(\xi)\|d\xi \right|\
\\
&\le C_{R_*}\left[ \left|t_1-t_2\right|+
  | u_{t_1}-u_{t_2}|_C^2\right]^{1/2} \le  C_{R_*}
\left|t_1-t_2\right|^{1/2}
\end{align*}
 for every $t_1,t_2\ge t_B\ge r$.
Thus from \eqref{main-eq} and \eqref{holder1} we obtain
\begin{equation*}
    ||A^{-\beta}( \dot u(t_1)-\dot u(t_2))||
    \le C_{R_*}|t_1-t_2|^{\al}~~\mbox{for all}~~ t_i\ge t_B,~~
    |t_1-t_2|\le 1,
\end{equation*}
for every $0<\al<1/2$.
This implies  that  the set   $D_{\al,\beta}^R$ given by \eqref{dis-comp}
 is absorbing for some $R$ provided
$0<\beta\le 1$ and $\al<\min\{\beta,1/2\}$.
\end{proof}

Proposition \ref{pr:dis-comp} allow us to apply the result
from \cite{PataZelik-cls} (see Theorem~\ref{th:cl-sem-attract}
in the Appendix) to guarantee the existence of a compact connected
global attractor.

\subsection{Dimension and exponential attractor}

The proof of finite-dimensionality is
based on the notion of {\it quasi-stability} which says that
the semigroup is asymptotically contracted   up to a homogeneous compact additive term.
 For the convenience we remind the  corresponding abstract result in the
Appendix.
\par
We can assume that there exists a \emph{ forward invariant} closed
absorbing set $D_0$ which belongs to $D_{\al,\beta}^R$ for an
appropriate choice of the parameters
 (see Proposition \ref{pr:dis-comp}).
 We also note that the restriction
of $S_t$ on $D_0$ is \textit{continuous} in both $t$ and initial data
 in the topology induced by ${\cCL }$
 (see
\eqref{sdd9-9}).
Thus
a dynamical system $(S_t, D_0)$ in the classical
(see \cite{Babin-Vishik,Chueshov_Acta-1999_book,Ha88,Temam_book}) sense arises.
Therefore we can apply the quasi-stability method developed earlier
in
\cite{Chueshov-2014_book,chlJDE04,Chueshov-Lasiecka-MemAMS-2008_book,Chueshov-Lasiecka-2010_book,CL-hcdte-notes}
 for continuous  evolution models.

\begin{proposition}[Quasi-stability]\label{pr:qs-main}
   Let Assumptions~\ref{as:basic}  and \ref{as:G2} be in force.
   Assume that \eqref{F0-fd} and \eqref{G-fd} are valid.
   Then
\begin{align}\label{qs-eq}
|S_t\va^1- S_t\va^2|_\cCL \le & C_R e^{-\lambda_1 t}\left[
||\va^1(0)- \va^2(0)||_{1/2}+
 |\va^1- \va^2|_{C}\right] \notag\\ &
 +C_{R} \max_{s\in [0,t]} ||A^{1/2-\beta}(u^1(s)-u^2(s))||,~~ t\ge r,
\end{align}
for every $\va^i\in D_0$,  where $u^i(t)=(S_t\va^i)(\theta)\big|_{\theta=0}$.
\end{proposition}
\begin{proof}
  Using the mild form presentation for $u^i(t)$
 and \eqref{G-fd}
we have that
  \begin{multline*}
 ||A^{1/2}( u^1(t)-u^2(t))||\le  e^{-\lambda_1 t}||A^{1/2}(u^1(0)- u^2(0))||
\\  + \int^t_0 ||A^{1-\delta} e^{-A(t-\tau)}||\cdot
\left( C || A^{-1/2+\delta}\big[ F(u^1_\tau)-F(u^2_\tau)\big] ||
+C_R
 || u^1(t)-u^2(t)||_{1/2-\gamma}\right)
\, d\tau.
\end{multline*}
As in \eqref{F-lip-l2}
we also have that
 \begin{align*}
\|  A^{-1/2+\delta}\big[F(u^2_t)-F(u^1_t)\big]\| & \le C \left|
\int_{t-\eta(u^1_t)}^{t-\eta(u^2_t)} \|A^{-\beta}\dot
u_2(\xi)\|d\xi\right| +  C| u^2_t- u^1_t|_C
\\ &  \le  C(R)
\max_{\theta\in [-r,0]} \| u^2(t+\theta)- u^1(t+\theta)\|
 \end{align*}
 for every $t\ge 0$.
Therefore
 \begin{multline*}
 ||A^{1/2}( u^1(t)-u^2(t))||\le  c_1e^{-\lambda_1 t}
 \left[
 ||A^{1/2}(\va^1(0)- \va^2(0))|| + |\va^1- \va^2|_{C}\right]
\\  + c_2(R)  \max_{s\in [0,t]} ||A^{1/2-\beta}(u^1(s)-u^2(s))||,~
\end{multline*}
Using \eqref{main-eq}, \eqref{G-lip} and \eqref{sdd17-F1} we also have that
\[
||A^{-1/2}( \dot u^1(t)-\dot u^2(t))||\le C(R) \left[||A^{1/2}( u^1(t)-u^2(t))||+
| u^2_t- u^1_t|_C\right]
\]
Thus
 \begin{multline*}
||A^{-1/2}( \dot u^1(t)-\dot u^2(t))|| \le  c_1e^{-\lambda_1 t}
 \left[
 ||A^{1/2}(\va^1(0)- \va^2(0))|| + |\va^1- \va^2|_{C}\right]
\\  + c_2(R)  \max_{s\in [0,t]} ||A^{1/2-\gamma}(u^1(s)-u^2(s))||,~~
t\ge r.
\end{multline*}
This completes the proof of Proposition~\ref{pr:qs-main}.
\end{proof}
In order to prove the finite dimensionality of the attractor $\fA$
we apply Theorem~\ref{th:qs-abstract} on the attractor with an appropriate choice
of operators and spaces.
Indeed,
let  $T>0$ be chosen such that  $q\equiv C_R e^{-\lambda_1 T}<1$
where $C_R$ is the constant from \eqref{qs-eq}.  We define the Lipschitz mapping $$
K: D_0 \mapsto Z_{[0,T]}\equiv
C^1([0,T];D(A^{-1/2}))\cap C([0,T];D(A^{1/2}))
$$  by the rule
$K\varphi = u(t), t\in [0,T]$, with $u$ be the unique solution of
(\ref{main-eq}) and (\ref{sdd9-ic}) with initial function $\varphi\in
D_0$.
The seminorm $n_Z (u)\equiv \max_{s\in [0,T]} ||A^{1/2-\beta}u(s)||$
is compact on $Z_{[0,T]}$ due to the compact imbedding of
$Z_{[0,T]}$  into $C([0,T];D(A^{\alpha}))$
by the Arzel\`{a}-Ascoli theorem (see, e.g.,\cite{Simon-AMPA-1987}).
\par
If we take
\begin{equation*}
Y \equiv \left\{ \varphi \in C^1([-r,0];H_{-1/2})\cap C([-r,0];H) \left|
 \varphi (0)\in H_{1/2} \right.\right\}
\end{equation*}
equipped with the norm \eqref{x-norm} and suppose $V=S_T$, then the
(discrete) quasi-stability inequality in \eqref{7.3.3h-g} is valid
on $D_0$. Hence we can apply Theorem 3.1.20
\cite{Chueshov-2014_book} (see Theorem~\ref{th:qs-abstract}) with
$V=S_T$,  $M= \fA$ and the quasi-stability estimate (\ref{qs-eq}) on
the attractor $\fA$ which lies in $D_0$. Thus ${\rm dim}_f\,\fA $ is
finite (in $X$ and thus in $\cCL$).
\medskip\par

To prove  the existence of a fractal exponential attractor
 we first use \eqref{7.3.3h-g} on the set $D_0$ and then apply  Theorem \ref{t7.3.2a-gen} to show that there exists a finite-dimensional set
 $A_\theta\subset D_0$ such that \eqref{7.3.3g}   holds.
Then as in the standard construction (see, e.g.,
\cite{EFNT94}  or \cite{MirZel-08}) we suppose
\[
\fA_{exp}=\overline{ \cup \{ S_t A_\theta\, : t\in [0,T]\}}.
\]
Since $V=S_T$ it is easy to see that $\fA_{exp}$ is exponentially
attracting, see \eqref{7.3.4} in the Appendix.

Since $D_0$ is included in the set $D^R_{\al,\beta}$ given by
\eqref{dis-comp}, we have that
$t\mapsto S_t\va$ is $\al$-H\"{o}lder on $D_0$ and
\[
|S_{t_1} \va-S_{t_2}\va|_{Y}\le C_{D_0} |t_1-t_2|^\al,
~~t_1,t_2\in [0,T],~~\va \in D_0.
\]
Therefore in the standard way
(see, e.g.,
\cite{EFNT94}  or \cite{MirZel-08})
we can conclude that $\fA_{exp}$ has  finite fractal dimension in $Y$.

This completes the proof of Theorem~\ref{th:attr}.

\medskip

{\bf Acknowledgements.} This work was supported in part by GA
 CR under project P103/12/2431.

\appendix
\section{Appendix}

Here, for the convenience of the reader, we remind some results used
in our work. For more details we refer to the cited sources.
\par
First
we collect some definitions and  properties, connected to (closed) evolution semigroups. We start with the following notion which was introduced in \cite{PataZelik-cls}.

\begin{definition}[Closed semigroup]
\label{Def-cl-semigr}  Let ${\cal
X}$ be a complete metric space. A {\tt closed semigroup} on ${\cal
X}$ is a one-parameter family of (nonlinear) operators $S_t : {\cal
X} \to {\cal X}\, (t\in \R_+)$ (or $t\in \Nset$) satisfying the conditions
\begin{itemize}
    \item[\textbf{(S.1)}] $S_0=Id_{\cal X}$-identical operotor;
    \item[\textbf{(S.2)}] $S_{t+\tau} = S_tS_\tau$ for all $t,
    \tau \in \R_+$;
    \item[\textbf{(S.3)}] for every $t\in\R_+$ the relations
     $x_n\to x$ and $S_tx_n \to y$ imply that
    $S_tx=y$.
\end{itemize}
\end{definition}
Assumptions (S.1) and (S.2) are the semigroup properties, while (S.3)
says that $S_t$ is a closed (nonlinear) map. We note  the operator closeness is a well-known concept in the theory of linear (unbounded) operators.
To our best knowledge
in the context of evolution operators this notion
was appeared in \cite{Babin-Vishik}   as a (weak) closeness of an evolution
(strongly continuous) semigroup
(see also \cite{Chueshov_Acta-1999_book}).

\par
The following notions are standard in the theory of infinite-dimensional evolution semigroups and dynamical systems
(see, e.g., \cite{Babin-Vishik,Chueshov_Acta-1999_book,Ha88,Ladyzhenskaya-1991-book,Temam_book}).

\begin{definition}[Dissipativity and compactness]
\label{Def-diss}
A semigroup $S_t$ is {\tt dissipative} if there is a bounded
absorbing set ${\cal B}_{abs}\subset {\cal X}$. That means for any
bounded set $B\subset {\cal X}$, there exists  $t_0=t_0(B)$ (the
entering time) such that $S_tB\subset {\cal B}_{abs}$ for all $t\ge
t_0$.   A semigroup $S_t$ is {\tt compact} if there is a compact
absorbing set.
\end{definition}

\begin{definition}[Global attractor]
\label{Def-attr2}
 A
{\tt global attractor}  of an evolution semigroup  $S_t$ acting
on a complete metric space ${\cal X}$ is defined as a bounded closed
set $\fA\subset {\cal X}$ which is  invariant ($S_t\fA=\fA$ for all
$t>0$)
and 
attracting.
 \end{definition}
 We recall
(\cite{Babin-Vishik,Temam_book}) that a
 set ${\cal K}\subset {\cal X}$ is called {\tt attracting} for
$S(t)$ if, for any bounded set $B\subset {\cal X}$,
$$ \lim_{t\to
+\infty}  d_{\cal X} \{ S(t)B, {\cal K}\} =0,$$ where $d_{\cal X} \{
A, B\} \equiv \sup_{x\in A} \, dist_{\cal X} (x,B)$ is the Hausdorff
semi-distance between bounded sets $A,B\subset~{\cal X}$.

%
%
%
%

The following assertion is a reformulation of Corollary 6 \cite{PataZelik-cls} which also takes into account the statement of \cite[Theorem 2]{PataZelik-cls}).
\begin{theorem}[Existence of a global attractor]\label{th:cl-sem-attract}
  Assume that $S_t : {\cal
X} \to {\cal X}$ is  a closed semigroup possessing a compact connected
 absorbing set ${\cal K}_{abs}\subset {\cal X}$.
Then there exists a compact global attractor $\fA$
for $S_t$. This attractor is a connected set and
  $\fA=\omega({\cal K}_{abs}) = \bigcap_{t \in R}
\overline{\bigcup_{\tau\ge t} S_\tau  {\cal K}_{abs}} $.
\end{theorem}

One of the desired qualitative properties of an attractor is its
finite-dimensionality. We remind the following definition.

\begin{definition}
\label{Def-fract} \cite{Chueshov_Acta-1999_book,Temam_book}).
{\it Let $M\subset {\cal X}$ be a compact set. Then the {\tt fractal
(box-counting) dimension} $dim_f M$ of $M$ is defined by
$$ dim_f M = \lim \sup_{\varepsilon \to 0} \frac{\ln n(M,\varepsilon)}{ \ln
(1/\varepsilon)},
$$
where $n(M,\varepsilon)$ is the minimal number of closed balls of
the radius $\varepsilon$ which cover the set $M$.
 }
\end{definition}

Our proof of the finite-dimensionality of the global attractor used
the following abstract result.

\begin{theorem}\label{th:qs-abstract}
 {\rm \textbf{  (\cite[Theorem 3.1.20]{Chueshov-2014_book})}}.
 {\it Let $Y$ be a Banach space and $M$ be a
bounded closed set in $Y$. Assume that there exists a mapping $V : M
\to Y$
such that \\
(i) $M \subset VM.$\\
 (ii) There exist a Lipschitz mapping $K$ from $M$
into some Banach space $Z$ and a compact seminorm $n_Z(x)$ on $Z$
such that
\begin{equation}\label{7.3.3h-g}
\| Vv^1-Vv^2\|\le \gamma \| v^1-v^2\|+  n_Z(Kv^1-Kv^2)
\end{equation}
 for any $v^1,v^2 \in M$, where $0 <\gamma < 1$ is a
constant. Then $M$ is a compact set in $Y$ of a finite fractal
dimension and
$$dim_f M \le \left[ \ln {2\over 1+\gamma} \right]^{-1} \cdot \ln m_Z \left( {4L_K \over 1 - \gamma} \right) ,
$$ where $L_K > 0$ is the Lipschitz constant for $K$:
$$ ||Kv^1-Kv^2||_Z \le L_K ||v^1-v^2||, v^1, v^2 \in M, 
$$
 and $m_Z(R)$
is the maximal number of elements $z_i$ in the ball $\{z \in Z :
||z_i||_Z \le R \}$ possessing the property $n_Z(z_i-z_j) > 1$ when
$i \neq j$. }
\end{theorem}

We recall (see \cite{EFNT94}) that a compact set $\fA_{\rm
exp}\subset \cCL$ is said to be  {\tt fractal exponential attractor}
for $S_t$  iff $\fA_{\rm exp}$ is a positively invariant set whose
fractal dimension is finite  and for every bounded set $D$ there
exist positive constants $t_D$, $C_D$ and $\gamma_D$ such that
\begin{equation}\label{7.3.4}
\sup_{\va\in D} \mbox{dist}\,_\cCL (S_t\va,\, \fA_{\rm exp})\le
C_D\cdot e^{-\gamma_D(t-t_D)}, \quad t\ge t_D.
\end{equation}
For  details  concerning fractal exponential attractors in the case
of continuous semigroups. we refer to \cite{EFNT94} and also to
the recent survey \cite{MirZel-08}.
We only mention that (i) a global attractor can be non-exponential
and (ii) an exponential attractor is not unique and   contains the
global attractor.

The dimension
theorem discussed above  pertains to
negatively or strictly invariant sets  $M$ ($M \subseteq V(M)$).
To prove the existence of exponential attractors we need an analog of Theorem~\ref{th:qs-abstract}  for  positively invariant sets.
More precisely we  need
 the following assertion which was established in \cite{Chueshov-2014_book}
 and is a version of the result
proved in \cite{Chueshov-Lasiecka-MemAMS-2008_book}   for metric spaces.
\begin{theorem}\label{t7.3.2a-gen}
Let $V\, :\, M\mapsto M$ be  a  mapping defined on  a
 closed bounded set $M$ of a Banach space $Y$.
Assume that
 there exist   a Lipschitz mapping $K$ from  $M$
into some Banach space $Z$ and a compact seminorm $n_Z(x)$  on $Z$ such that the property in \eqref{7.3.3h-g} holds.
Then for any $\theta\in (\ga, 1)$ there exists a positively invariant
compact set $A_\theta\subset M$ of finite fractal dimension  satisfying
\begin{equation}\label{7.3.3g}
\sup\left\{ {\rm dist} (V^ku, A_\theta)\; :\; u\in M\right\}\le r\theta^k,
\quad k=1,2,\ldots,
\end{equation}
for some constant $r>0$.
Moreover,
\begin{equation*}
\dim_fA_\theta\le
\ln m_Z\left(
\frac{2L_K}{\theta-\ga}\right).\left[\ln\frac{1}{\theta}\right]^{-1}
\,,
\end{equation*}
where we use the same notations  as in Theorem~\ref{th:qs-abstract}.
\end{theorem}

\end{document}